\documentclass[12pt]{article}

\usepackage{amsthm,amsmath,amssymb}

\usepackage{graphicx}

\usepackage{mathrsfs}

\theoremstyle{plain}
\newtheorem{theorem}{Theorem}
\newtheorem{lemma}[theorem]{Lemma}
\newtheorem{corollary}[theorem]{Corollary}
\newtheorem{proposition}[theorem]{Proposition}

\theoremstyle{definition}

\newtheorem{problem}[theorem]{Problem}

\theoremstyle{remark}

\newcommand{\set}[1]{\left\{ #1 \right\}}
\newcommand{\setdef}[2]{\left\{ #1 \colon\ #2 \right\}}

\title{\bf Panchromatic patterns by paths}

\author{Germ\'an Ben\'itez-Bobadilla\\
\small Instituto de Matem\'aticas\\[-0.8ex]
\small UNAM\\[-0.8ex]
\small Mexico City, Mexico\\
\small\tt gbenitez@matem.unam.mx\\
\and
Hortensia Galeana-S\'anchez\\
\small Instituto de Matem\'aticas\\[-0.8ex]
\small UNAM\\[-0.8ex]
\small Mexico City, Mexico\\
\small\tt hgaleana@matem.unam.mx\\
\and
C\'esar Hern\'andez Cruz\\
\small Departamento de Computaci\'on\\[-0.8ex]
\small Centro de Investigaci\'on y de Estudios Avanzados del IPN\\[-0.8ex]
\small Mexico City, Mexico\\
\small\tt cesar@cs.cinvestav.mx
}

\date{\today }

\begin{document}

\maketitle

\begin{abstract}
Let $H=(V_H,A_H)$ be a digraph, possibly
with loops, and let $D=(V_D, A_D)$ be a
loopless multidigraph with a colouring of
its arcs $c: A_D \rightarrow V_H$. An
$H$-path of $D$ is a path $(v_0, \dots,
v_n)$ of $D$ such that $(c(v_{i-1}, v_i),
c(v_i,v_{i+1}))$ is an arc of $H$
for every $1 \le i \le n-1$. For $u, v
\in V_D$, we say that $u$ reaches $v$ by
$H$-paths if there exists an $H$-path from
$u$ to $v$ in $D$. A subset $S \subseteq
V_D$ is $H$-absorbent of $D$ if every
vertex in $V_D-S$ reaches by $H$-paths
some vertex in $S$, and it is $H$-independent
if no vertex in $S$ can reach another
(different) vertex in $S$ by $H$-pahts.
An $H$-kernel is an independent by $H$-paths
and absorbent by $H$-paths subset of $V_D$.

We define $\tilde{\mathscr{B}}_1$ as the
set of digraphs $H$ such that any
$H$-arc-coloured tournament has an
$H$-absorbent by paths vertex; the set
$\tilde{\mathscr{B}}_2$ consists of the
digraphs $H$ such that any $H$-arc-coloured
digraph $D$ has an independent,
$H$-absorbent by paths set; analogously,
the set $\tilde{\mathscr{B}}_3$ is the set
of digraphs $H$ such that every
$H$-arc-coloured digraph $D$ contains an
$H$-kernel by paths.

In this work, we present a characterization
of $\tilde{\mathscr{B}}_2$, and provide
structural properties of the digraphs in
$\tilde{\mathscr{B}}_3$ which settle up its
characterization except for the analysis of
a single digraph on three vertices.

  \bigskip

\end{abstract}

\section{Introduction}

A kernel in a digraph $D$ is a subset $K$
of $V_D$ which is independent and
absorbent\footnote{The set $K$ is absorbent
if for every vertex $u \in V_D - K$ there
exists a vertex $v$ in $K$ such that $(u,v)$
is an arc of $D$.}.   Numerous generalizations
of this concept have been studied in the
literature, but one that has received particular
attention is the notion of a kernel by
monochromatic walks.   Given a digraph $D$
whose arc set is coloured with a set of colours
$C$, we say that a vertex $u$ reaches a vertex
$v$ by monochromatic walks in $D$ if there is a
(directed) walk from $u$ to $v$ such that all its
arcs have the same colour.  With this generalized
reachability concept, the notions of independence
and absorbance by monochromatic walks come
naturally; as one would expect, a subset $S$ of
$V_D$ is independent by monochromatic walks if no
vertex in $S$ can reach another vertex in $S$ by
monochromatic walks, and it is absorbent by
monochromatic walks if every vertex in $V_D - S$
reaches some vertex in $S$ by monochromatic walks.
Evidently, a kernel by monochromatic walks is a
subset $K$ of $V_D$ which is independent and
absorbent by monochromatic walks.   Sands, Sauer
and Woodrow proved in \cite{sandsJCTB33} that
every digraph whose arc set is coloured with two
colours has a kernel by monochromatic walks.

In \cite{linekAC44}, Linek and Sands further
generalized kernels by monochromatic walks by
considering a broader notion of reachability.
Instead of colouring the arcs of the digraph
$D$ with an arbitrary set, they used the vertex
set of another digraph $H$; with this setting,
instead of only considering monochromatic walks,
the arcs of $H$ could be used to codify permitted
colour changes in the walks of $D$ to define
reachability.   Formally, a digraph $H$, possibly
with loops will be called a {\em pattern of
colours}, or {\em pattern} for short, and $D$
is an {\em $H$-arc-coloured digraph} if it is an
irreflexive multidigraph together with a colouring
$c$ of its arcs, $c\colon A_D \to V_H$.   An
{\em $H$-walk} $W$ in $D$ is a walk $W = (x_0,
\dots, x_k)$ such that $(c(x_0, x_1), \dots,
c(x_{k-1}, x_k))$ is a walk in $H$.   For $u,
v \in V_D$, we say that $u$ reaches $v$ by
$H$-walks if there exists an $H$-walk from $u$
to $v$ in $D$.   Naturally, this notion of
reachability allows us to define independence
and absorbance by $H$ walks.   A susbset $S$ of
$V_D$ is independent by $H$-walks if no vertex in
$S$ can reach by $H$-walks another vertex in $S$,
and it is absorbent by $H$-walks if every vertex
in $V_D - S$ can reach by $H$-walks some vertex
in $S$; is $S$ is both independent by $H$-walks
and absorbent by $H$-walks, we say that $S$ is a
kernel by $H$-walks.   Notice that if the only
arcs of $H$ are loops, then the only possible
$H$-walks are precisely monochromatic walks, and
hence, a kernel by $H$-walks is just a kernel by
monochromatic walks.

In this context, Arpin and Linek \cite{arpinDM307}
introduced three interesting classes of patterns.
The family $\mathscr{B}_1$ is the class of patterns
$H$ with the property that for every $H$-arc-coloured
tournament $T$ there exists a vertex which is
$H$-absorbent by $H$-walks. Noticing that every
independent set in a tournament has a single vertex,
we may consider a special subclass of $\mathscr{B}_1$;
the family $\mathscr{B}_2$ which consist of the patterns
$H$ with the property that for every $H$-arc-coloured
digraph $D$ there exists an $H$-absorbent set by walks
which is an independent set.   Further noticing that
every independent by $H$-walks set is independent, the
family $\mathscr{B}_3$ of the patterns $H$ such that
every $H$-arc-coloured digraph contains a kernel by
$H$-walks, seems to be a nice subfamily of
$\mathscr{B}_2$ to consider.   The patterns in
$\mathscr{B}_3$ are known as {\em panchromatic
patterns} (by walks).   In the same article, Arpin
and Linek characterized the class $\mathscr{B}_2$; it
consists of those patterns $H$ with no odd directed
cycle in their complement.

In \cite{galeanaDM339} Galeana-S\'anchez and Strausz
presented a characterization of the class
$\mathscr{B}_3$.   Nonetheless, the authors of the
present work recently found that the proof of a crucial
lemma in the aforementioned work is incorrect, and thus,
the characterization of this class remains as an open
problem.   Galeana-Sánchez and Hern\'andez-Cruz used
this flawed characterization in \cite{galeanaJGT90} to
prove that every pattern $H$ is either a panchromatic
pattern, or the problem of determining whether an
$H$-arc-coloured digraph has a kernel by $H$-walks is
$NP$-complete; luckily, the main argument is robust
enough to remain true even if the family of panchromatic
patterns by $H$-walks turns out to be a family different
from the one proposed in \cite{galeanaDM339}.

All the developements mentioned so far use a notion of
reachability that relies on the concept of $H$-walks.
It is easy to notice that every monochromatic walk contains
a monochromatic path, so it comes to mind to consider
reachability by $H$-paths instead of reachability by
$H$-walks.   Of course, this new notion of reachability
yields new notions of independence, absorbance and kernels,
but also, analogous to the families $\mathscr{B}_i$ might
be defined for $i \in \set{1, 2, 3}$ using reachability by
$H$-paths.   So, the aim of the present work is to study
these concepts, similarities and differences with the
reachability by $H$-walks case.   To avoid the cumbersome
nomenclature, a kernel by $H$-walks will be simply called
an $H$-kernel, and, analogously, $H$-independence
and $H$-absorbance will refer to independence
by $H$-paths and absorbance by $H$-paths,
respectively.

It is worth noticing that reachability problems
in arc-coloured digraphs (and tournaments in
particular) have received a a lot of attention
and many variants are being studied, consider,
for an instance, \cite{beaudouDM342,
bousquetArXiv, delgadoDAM236, hahnDM283}.

We refer to \cite{bang2002} and \cite{bondy2008}
for general concepts.  A digraph is
{\em reflexive} if every vertex has a loop.
Let $H$ be a reflexive pattern. Two different
vertices  $x$ and $y$ of $H$ are {\em true
twins} if $N^+(x) = N^+ (y)$ and $N^-(x) =
N^-(y)$.   Notice that if follows from the
previous definition that, if $x$ and $y$ are
true twins, then we obtain isomorphic patterns
if we delete either of them or we contract
the set $\{ x, y \}$ and delete any parallel
arcs that were created from this contraction.
We will denote this kind of contraction of
the set $\{ x, y \}$ (without parallel arcs)
as $H_{xy}$.   From the previous observation
it is clear that $H_{xy}$ is isomorphic to an
induced subdigraph of $H$.

The rest of the paper is organized as follows.
In Section \ref{sec:basic} some basic results
regarding the differences and similarities between
reachability by $H$-walks and by $H$-paths are
presented, in particular, for a pattern $H$ we
propose constructions to find infinite families
of $H$-coloured digraphs having $H$-kernel but no
kernel by $H$-walks, and vice versa.   Section
\ref{sec:b2} is devoted to prove that the class
$\tilde{\mathscr{B}_2}$, the $H$-path reachability
analogue to $\mathscr{B}_2$, is actually equal to
$\mathscr{B}_2$.   In Section \ref{sec:trans}, we
study the patterns $H$ such that for every
$H$-arc-coloured digraph, every $H$-walk between
two vertices of $D$ contains an $H$-path with
the same endpoints.   Sections \ref{sec:ppp} and
\ref{sec:F4F5} are devoted to panchromatic
patterns by paths; in the former, we analyze
thirteen of the patterns on three vertices, and
in the latter we deal with two of the remaining
three.   In Section \ref{sec:F1} we point out
some properties of the pattern $F_1$, which is
the only one we were unable to decide whether
it is as a panchromatic pattern by paths.
Section \ref{sec:Conc} closes this work by
presenting conclusions and some open problems.

\section{Basic results}
\label{sec:basic}

As we mentioned before, our main goal is
to obtain a development parallel to that
made in \cite{arpinDM307}, but considering
reachability by $H$-paths instead of
reachability by $H$-walks.   Hence, we
start defining the following classes of
digraphs.

\begin{itemize}
    \item Let $\tilde{\mathscr{B}_1}$ the class of patterns $H$
        such that every $H$-arc-colourerd tournament contains
        an $H$-kernel.

    \item Let $\tilde{\mathscr{B}_2}$ the class of patterns $H$
        such that every $H$-arc-coloured digraph $D$, contains
        an independent set of $V_D$ which is also $H$-absorbent.

    \item Let $\tilde{\mathscr{B}_3}$ the class of patterns $H$
        such that for every $H$-arc-coloured digraph $D$ contains
        an $H$-kernel.
\end{itemize}

By simply considering the definitions
it is easy to observe that
$\tilde{\mathscr{B}_3}  \subseteq
\tilde{\mathscr{B}_2}\subseteq
\tilde{\mathscr{B}_1}$.   Notice that
if $H$ is a complete reflexive digraph,
then $H \in \tilde{\mathscr{B}_3}$.
Sands, Sauer and Woodrow proved in
\cite{sandsJCTB33} that if $H$ consists
of two reflexive isolated vertices,
then $H \in \tilde{\mathscr{B}_3}$,
hence $\tilde{\mathscr{B}_i}$ is nonempty
for $i \in \set{1, 2, 3}$.   Naturally,
we are interested in a characterization
of these classes; our following result states
some simple observations about them.

\begin{lemma}\label{prop Bi}
Let $H$ be a pattern.
\begin{enumerate}
    \item If $H \in \tilde{\mathscr{B}}_i$,
          then every vertex of $H$ is reflexive,
          for $i \in \set{1, 2, 3}$.

    \item If $H \in \tilde{\mathscr{B}}_i$
          and $H_1$ is an induced subdigraph
          of $H$ then $H_1 \in
          \tilde{\mathscr{B}}_i$, for $i
          \in \set{1, 2, 3}$.

    \item If $H \in \tilde{\mathscr{B}}_i$
          and $H_1$ is a spanning
          superdigraph of $H$, then $H_1 \in
          \tilde{\mathscr{B}}_i$, for $i \in
          \set{1, 2}$.
\end{enumerate}
\end{lemma}

\begin{proof}
For the first item, consider a pattern
$H$ with a vertex $x$ such that $(x,x)
\notin A_H$, then the $H$-arc-coloured
digraph $D$ obtained by colouring each
arc of $C_3$ with $x$ does not have an
$H$-kernel, and thus, $H \notin
\tilde{\mathscr{B}}_i$ for $i \in \set{
1, 2, 3}$.

For (2), since every $H_1$-arc-colouring
is also an $H$-arc-colouring, then any
example showing $H_1 \notin
\tilde{\mathscr{B}}_i$ also shows that
$H \notin \tilde{\mathscr{B}}_i$, for
$i \in \set{1, 2, 3}$.

For (3) we have that any $H$-path is also
an $H_1$-path, thus $H_1 \in
\tilde{\mathscr{B}}_i$, for $i \in \set{
1, 2}$.
\end{proof}

Since the existence of an $H$-walk between
two vertices does not guarantee the existence
of an $H$-path between those vertices and
the concatenation of two $H$-paths is not
always an $H$-walk, we claim that if $D$
has a kernel by $H$-walks, then $D$ not
necessarily has an $H$-kernel, as the
example in Figure \ref{Fig1} shows. In
Figure \ref{Fig1} we have that $\set{w}$
is a kernel by $H$-walks of $D$, because
$(x, y, z, y, w)$ is a spanning $H$-walk
in $D$ that finishes in $w$. It is easy
to check that $D$ has no $H$-kernel
(notice that every $H$-independent set
by paths of $D$ has cardinality one).

\begin{figure}[!htb]
\begin{center}
\includegraphics[scale=.35]{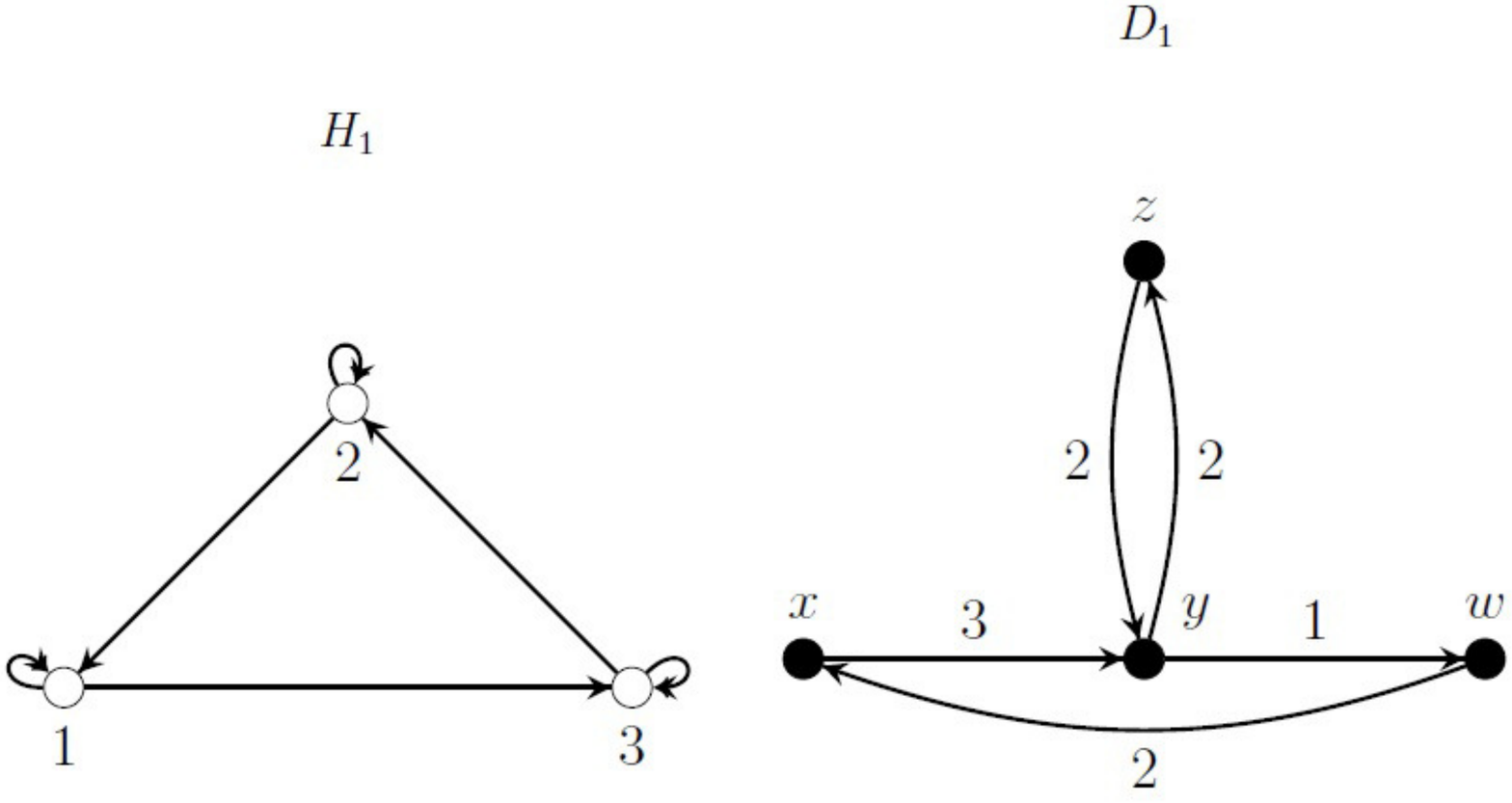}
\caption{The set $\set{w}$ is a kernel
by $H$-walks of $D$ and $D$ has no
$H$-kernel.}
\label{Fig1}
\end{center}
\end{figure}

We also claim that if $D$ has an
$H$-kernel, then $D$ not necessarily
has a kernel by $H$-walks as the example
in Figure \ref{Fig2} shows. In Figure
\ref{Fig2} we have that $\set{u, v}$
is an $H$-kernel of $D$. It is easy
to see that $D$ has no $H$-kernel by
walks (notice that every independent
by $H$-walks set in $D$ has cardinality
one because $(u, y, z, y, v)$ is an
$H$-walk between in $D$).

\begin{figure}[!htb]
\begin{center}
\includegraphics[scale=.35]{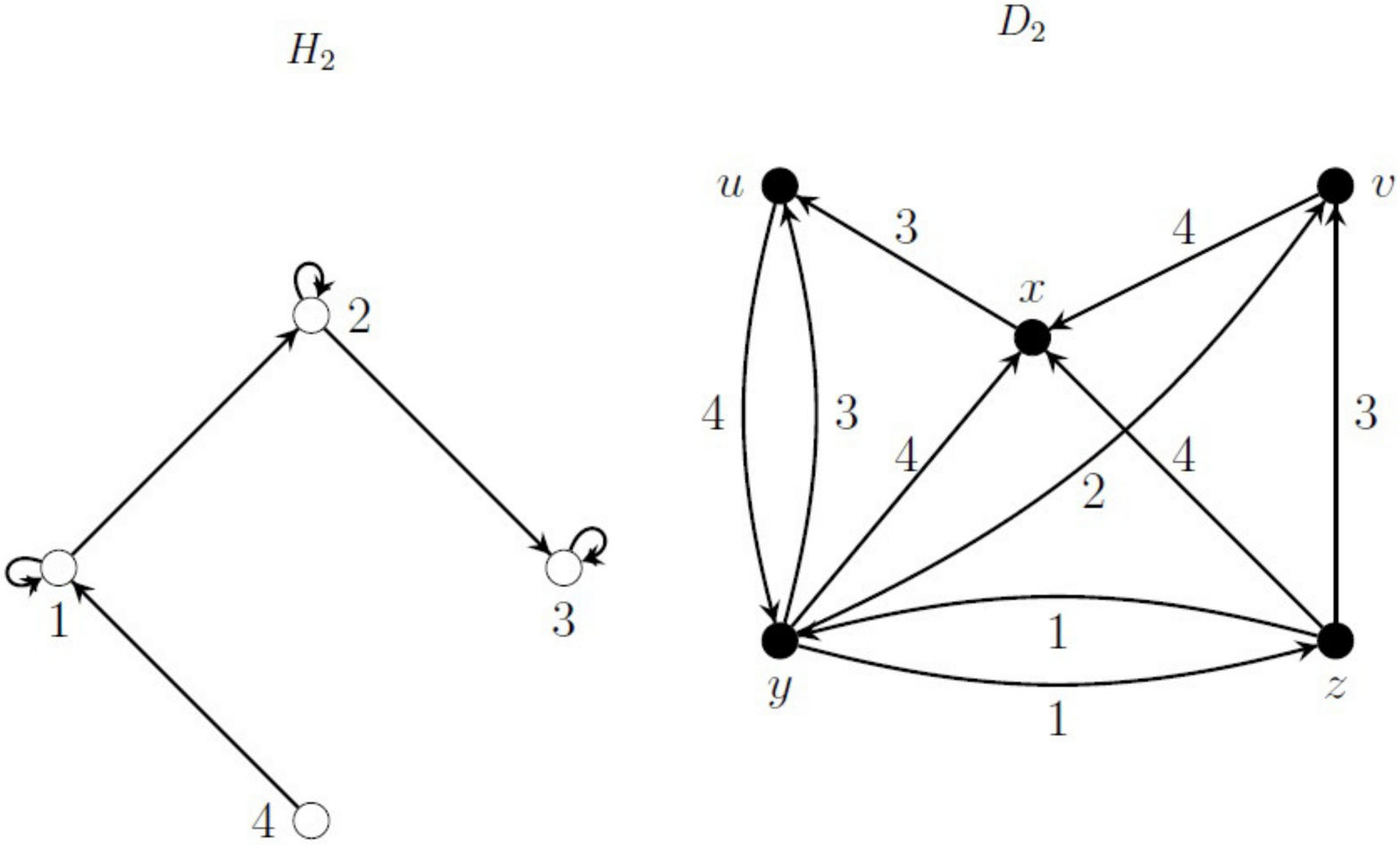}
\caption{$\set{u, v}$ is an $H$-kernel
of $D$ and $D$ has no kernel by $H$-walks.}
\label{Fig2}
\end{center}
\end{figure}

To convince the reader that these are not
isolated cases, we now exhibit inifinite
families with the same behaviour as the
previous two examples.

Given two multidigraphs $D_1$ and $D_2$,
we define the multidigraph $D_1\bullet D_2$,
which we call the {\em linear sum} of $D_1$
and $D_2$, to have vertex set $V_{D_1 \bullet
D_2} = V_{D_1} \cup V_{D_2}$ (disjoint union)
and arc set $A_{D_1 \bullet D_2} = A_{D_1}
\cup A_{D_2} \cup V_{D_1} \times V_{D_2}$.

Let $H$ be a pattern, and let $D_1$ and $D_2$
be two $H$-arc-coloured digraphs with
colourings $c_1$ and $c_2$, respectively.
Consider the linear sum of $D_1$ and $D_2$.
We define an $H$-arc-colouring $c$ of $D_1
\bullet D_2$ as follows, $c : A_{D_1 \bullet
D_2} \to V_H$ is such that \[
c(a)=
\left\{ \begin{array}{ll}
 c_1(a) & \text{ if }a\in A_{D_1} \\
 c_2(a) & \text{ if }a\in A_{D_2} \\
 c_0    & \text{ if } a = (u, v) \text{ with }
          u \in V_{D_1} \text{ and } v \in
          V_{D_2} \text{ where } c_0
          \text{ may or not be in } H.
\end{array}
\right.
\]
Based on the previous $H$-arc-colouring
we can obtain the following lemmas.

\begin{lemma}
\label{sumPathsNoWalks}
Let $H$ be a pattern.  If $D_1$ and $D_2$ are
$H$-arc-coloured digraphs, then $D_1\bullet D_2$
has an $H$-kernel and has no kernel by $H$-walks
if and only if $D_2$ has $H$-kernel and has no
kernel by $H$-walks.
\end{lemma}

\begin{proof}
Suppose first that $D_1 \bullet D_2$ has an
$H$-kernel. We will show that $D_2$ has an
$H$-kernel and has no kernel by $H$-walks.
Let $K$ be an $H$-kernel of $D_1 \bullet
D_2$. We will prove that $K$ is an
$H$-kernel of $D_2$.

Since every vertex in $V_{D_1}$ dominates
every vertex in $V_{D_2}$, it is not hard
to notice that $K \subseteq V_{D_2}$. On
the other hand, by the definition of linear
sum and its associated colouring $c$, it
follows that any $H$-path in $D_1 \bullet
D_2$ between two vertices $x$ and $y$ in
$V_{D_2}$, is also an $H$-path in $D_2$.
Hence, since $K$ is an $H$-kernel of $D_1
\bullet D_2$, then $K$ is $H$-kernel of
$D_2$.

We will prove that if $D_1 \bullet D_2$
does not have a kernel by $H$-walks,
then neither does $D_2$ by contrapositive.
Again, by the definition of $D_1 \bullet
D_2$, every vertex in $V_{D_1}$ dominates
every vertex in $V_{D_2}$ and no vertex in
$V_{D_2}$ dominates a vertex in $V_{D_1}$.
Also, the definition of $c$ guarantees
that an $H$-walk in $D_2$ is also an
$H$-path in $D_1 \bullet D_2$.   It
follows that a kernel by $H$-walks in
$D_2$ is also kernel by $H$-walks in $D_1
\bullet D_2$.

Analogously, an $H$-kernel of $D_2$ is
also an $H$-kernel of $D_1 \bullet D_2$.
So, proceeding again by contrapositive,
assume that $D_1 \bullet D_2$ has a
kernel by $H$-walks $K$.   Again, by
the structure of $D_1 \bullet D_2$, it
must be the case that $K \subseteq
V_{D_2}$.   Also, every $H$-walk in
$D_1 \bullet D_2$ between vertices of
$D_2$ is an $H$-walk in $D_2$.   Thus,
$K$ is a kernel by $H$-walks in $D_2$.
\end{proof}

The proof of our following lemma is
analogous to the proof of Lemma
\ref{sumPathsNoWalks}, so we omit it.

\begin{lemma}
\label{sumWalksNoPaths}
Let $H$ be a pattern.   If $D_1$ and $D_2$ are
$H$-arc-coloured digraphs, then $D_1 \bullet
D_2$ has a kernel by $H$-walks and has no
$H$-kernel if and only if $D_2$ has a kernel
by $H$-walks and has no $H$-kernel.
\end{lemma}

As noted above, using the example in Figure
\ref{Fig2} and Lemma \ref{sumPathsNoWalks}
we can recursively construct an infinite
family of digraphs having an $H$-kernel
and no kernel by $H$-walks as follows.
Let $H_2$ and $D_2$ be digraphs as in
Figure \ref{Fig2}. Define
\begin{itemize}
    \item $D^1 = D_2$.

    \item $D^{j+1} = K_1 \bullet D^j$.
\end{itemize}
The resulting family of digraphs clearly
has the desired properties.

Similarly, using the example in \ref{Fig1}
and Lemma \ref{sumWalksNoPaths} an infinite
family of digraphs having a kernel by
$H$-walks and has no $H$-kernel can be
recursively defined as follows. Let $H_1$
and $D_1$ be digraphs as in Figure
\ref{Fig2}.
\begin{itemize}
    \item $E^1 = D_1$.

    \item $E^{j+1} = K_1 \bullet E^j$.
\end{itemize}

Hence, the problem of determining the existence
of an $H$-kernel and the problem of determining
the existence of a kernel by $H$-walks are
indeed different problems.   In view of these
examples, the results of the following sections
come a bit as a surprise.


\section{The $\tilde{\mathscr{B}}_2$ class}
\label{sec:b2}

In this section we characterize the patterns that belong to
$\tilde{\mathscr{B}}_2$ as the reflexive digraphs such that
their complement do not contain odd directed cycles, i.e.,
$\tilde{\mathscr{B}}_2 = \mathscr{B}_2$.

\begin{lemma}\label{sumB2}
Let $H_1$ and $H_2$ be reflexive patterns. If $H_1$,
$H_2 \in \tilde{\mathscr{B}}_2$, then $H_1 \bullet H_2 \in
\tilde{\mathscr{B}}_2$.
\end{lemma}

\begin{proof}
Let $D$ be an $H_1 \bullet H_2$-arc-colured multidigraph.   For
$i \in \{ 1, 2 \}$, let $D_i$ be the spanning subdigraph of $D$
with arc set consisting of precisely those arcs of $D$ coloured
with the vertices of $H_i$.   Let $K$ be an independent
$H_1$-absorbent set of $D_1$, and let $K'$ be an independent
$H_2$-absorbent set of $D_2[K]$, the subdigraph induced by $K$
in $D_2$.   Clearly, $K'$ is independent in $D$; to verify
that $K'$ is $(H_1 \bullet H_2)$-absorbent, let $x$ be a
vertex of $D$ not in $K'$.   If $x \in K$, then the choice of
$K'$ guarantees that $x$ is $H_2$-absorbed by it.   If $x
\notin K$, then there is an $H_1$-path $P_1$ in $D_1$ (and
thus also in $D$) from $x$ to some vertex $y \in K$.   If $y
\notin K'$, then, as in the previous case, there is an
$H_2$-path $P_2$ from $y$ to some vertex $z \in K'$ in $D_1
[K]$.   Let $s$ be the last vertex in $P_2$ which is also in
$P_1$, then, by the structure of $H_1 \bullet H_2$, we have
that $x P_1 s P_2 z$ is an $(H_1 \bullet H_2)$-path in $D$.
Therefore, $K'$ is an independent $(H_1 \bullet
H_2)$-absorbent set.
\end{proof}

\begin{theorem}
\label{B2  characterization}
If $H$ is a reflexive pattern, then $H \in \tilde{\mathscr{B}}_2$
if and only if $H^c$ does not contain any odd cycles.
\end{theorem}

\begin{proof}
For the necessity, we prove the contrapositive. If $H^c$
contains the odd cycle $(0, 1, 2, \dots , 2k, 0)$ then let $D$
be the odd cycle $D = (x_0, x_1, \dots ,x_{2k}, x_0)$ and
$c(x_i, x_{i+1}) = i$, for $0 \le i \le 2k-1$, and $c(x_{2k},
x_0) = 2k$.   Every $H$-path in $D$ is a single arc, hence
$D$ has no independent $H$-absorbent set and $D \notin
\tilde{\mathscr{B}}_2$.

For the sufficiency, let us suppose that $H^c$ has no odd
cycle. Let $H_1,\dots, H_k$ be an ordering of the strong
components of $H^c$, such that $H_i$ is an  initial
component of $H^c-(\bigcup_{j=1}^{i-1} H_j)$. Since every
$H_i$ is strongly connected and has no odd cycles, then
it is a bipartite digraph with bipartition $(V_{i1}, V_{i2})$,
for $1 \le i \le k$. It follows that in $H$ every $V_{ij}$
induces a complete reflexive digraph, for $1\le i\le k$
and for $1\le j\le 2$. Then, by Lemma \ref{prop Bi}.3, $H_i
\in \tilde{\mathscr{B}}_2$.

 By definition of $H_i$ in $H$, $H' = H_k \bullet (\cdots
 \bullet (H_3 \bullet(H_2 \bullet H_1)) \cdots )$ is a spanning
 subdigraph of $H$. Then, it follows from the previous lemma
 that $H \in \tilde{\mathscr{B}}_2$.
\end{proof}


\section{Transitive patterns}
\label{sec:trans}

In this section we present a characterization
of the patterns $H$, such that for every digraph
$D$ and every $H$-arc-colouring of $D$, every
$H$-walk between two vertices in $D$ contains
an $H$-path with the same endpoints.

Given a set of digraphs $\mathcal{S}$, we say
that a digraph is {\em $\mathcal{S}$-free} if
it does not contain any digraph in $\mathcal{S}$
as an induced subdigraph.   Recall that if a
property $\mathcal{P}$ is closed under taking
induced subdigraphs, then there exists a family
of forbidden digraphs, $\mathcal{F}$, not
necessarily finite, that characterize the digraphs
with that property.   In other words, a digraph
satisfies the property $\mathcal{P}$ if and only
if it is $\mathcal{F}$-free. As an example,
consider the following theorem.

\begin{theorem}
\label{transdigchar}
If $H$ is a reflexive digraph, then $H$ is a
transitive digraph if and only if it is
$\mathcal{T}$-free, where $\mathcal{T}$ is
the family of reflexive digraphs depicted
in Figure \ref{transFor}.
\end{theorem}

\begin{proof}
The necessity is trivial.  The sufficiency
follows from a simple analysis of cases
considering two consecutive arcs in a
$\mathcal{T}$-free digraph $D$.
\end{proof}

If we consider the elements of $\mathcal{T}$
as colour patterns $H$, each of them is an
example showing that not every $H$-walk
contains an $H$-path with the same endpoints;
see Figure \ref{transFor}.   Our following
result states that these are the only
(minimal) patterns with this property.

\begin{figure}[!htb]
\begin{center}
\includegraphics[scale=.9]{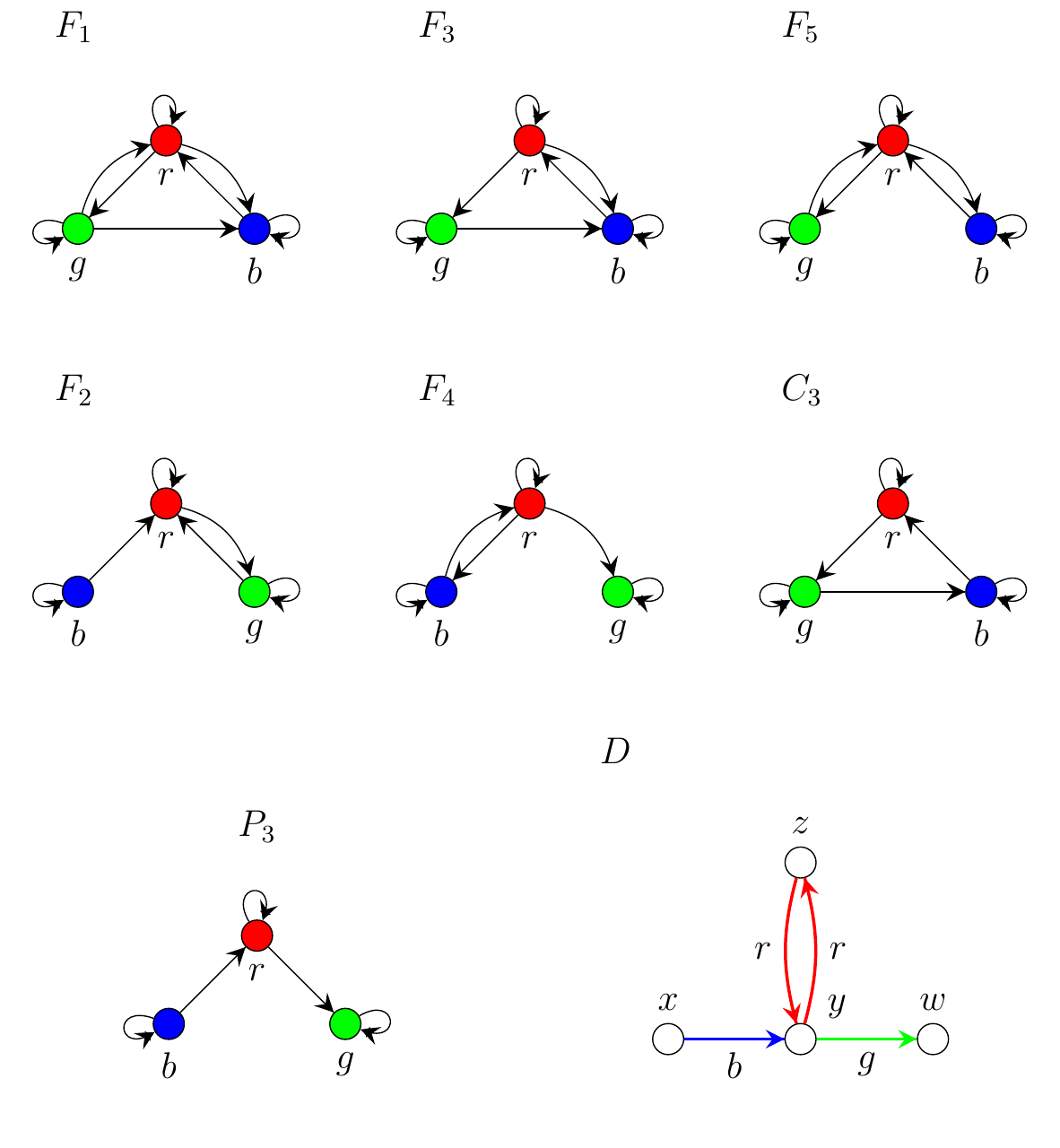}
\caption{The family $\mathcal{T}$, and an
example where $(x,y,z,y,w)$ is an $H$-walk
from $x$ to $w$ in $D$ but there is no
$H$-path from $x$ to $w$ in $D$, for every $H$
in $\mathcal{T}$.}
\label{transFor}
\end{center}
\end{figure}

\begin{theorem} \label{walkContainsPath}
Let $H$ be a reflexive digraph. The
following statements are equivalent.
\begin{enumerate}
    \item $H$ is transitive.

    \item $H$ is $\mathcal{T}$-free.

    \item For every $H$-arc-coloured
          digraph $D$, and every pair
          of different vertices $x,y$
          of $D$, every $H$-walk from
          $x$ to $y$ in $D$ contains an
          $H$-path from $x$ to $y$ in $D$.
\end{enumerate}
\end{theorem}

\begin{proof}
The equivalence of (1) and (2) has been
already proved in Theorem \ref{transdigchar}.

We will prove (3) implies (2) by
contrapositive.   Suppose that $H$ is a
reflexive digraph containing one of the
patterns in Figure \ref{transFor} as an
induced subdigraph.  Hence, the digraph
at the bottom right of the same figure
is an $H$-arc-coloured digraph with
vertices $x$ and $w$ such that an $H$-walk
exists from $x$ to $w$, but no $H$-path
exists with the same endpoints.

For (1) implies (3), suppose that $H$
is transitive reflexive digraph. Let
$D$ be an $H$-arc-coloured digraph,
let $x$ and $y$ be different vertices
in $D$ and let $\gamma = (x = x_0,
\dots, x_n = y)$ be an $H$-walk from
$x$ to $y$ in $D$. Proceeding by
induction on the length of $\gamma$,
$\ell (\gamma)$, we will prove that
$\gamma$ contains an $H$-path from
$x$ to $y$ in $D$.

If $\ell (\gamma)=1$, then $\gamma$ is
an $H$-path from $x$ to $y$ in $D$. Let
$\gamma = (x = x_0, x_1, \dots, x_n =
y)$ be an $H$-walk from $x$ to $y$ in
$D$ of length $n$.   If $x_i \ne x_j$
for each $i \neq j$, then $\gamma$ is an
$H$-path from $x$ to $y$ in $D$.  Else,
consider $i, j \in \{ 0, \dots, n \}$,
such that $x_i = x_j$ and $i < j$.
Since $\gamma$ is an $H$-walk in $D$,
we have that $(c(x_{k-1}, x_k), c(x_k,
x_{k+1})) \in A_H$ for each $k \in \{
1, \dots, n-1 \}$. Moreover, as
$(c(x_{i-1}, x_i), c(x_i,x_{i+1}),
\dots , c(x_{j-1}, x_j), c(x_j,
x_{j+1}))$  is a directed walk in $H$,
then $(c(x_{i-1}, x_i), c(x_j, x_{j+1}))
\in A_H$ by the transitivity of $H$.
Therefore $\gamma' = (x = x_0, x_1,
\dots, x_i = x_j, x_{j+1}, \dots, x_n
= y)$ is an $H$-walk form $x$ to $y$
in $D$ of length less than $n$. By
induction hypothesis, $\gamma'$
contains an $H$-path from $x$ to $y$
in $D$, say $T$ which is also contained
in $\gamma$.
\end{proof}

The above result allows us to establish
the following theorem.

\begin{theorem}\label{transPathIffWalk}
Let $H$ be a transitive reflexive digraph
and $D$ an $H$-arc-coloured digraph. The
subset $K$ of $V_D$ is a kernel by
$H$-walks of $D$ if and only if it is an
$H$-kernel of $D$.
\end{theorem}

\begin{proof}

Let $K \subseteq V_D$ be a kernel by $H$-walks of $D$.
Since $K$ is independent by $H$-walks in $D$ and
every $H$-path is an $H$-walk in $D$ then, in particular,
$K$ is an $H$-independent set in $D$. To show that $K$
is an $H$-absorbent set in $D$, consider a vertex $x$ in
$V_D  - K$. Since $K$ is absorbent by $H$-walks
in $D$, then there is a $w$ in $K$ such that there
is an $H$-walk from $x$ to $w$ in $D$, say $\gamma$.
By Theorem \ref{walkContainsPath} $\gamma$ contains
an $H$-path from $x$ to $w$ in $D$. Therefore $K$ is an
$H$-kernel of $D$.

Conversely, let $K \subseteq V_D$ be an $H$-kernel of $D$.
Every $H$-path is an $H$-walk in $D$, so we have that
$K$ is absorbent by $H$-walks.   It follows from Theorem
\ref{walkContainsPath} that every $H$-walk contains an
$H$-path; since  there are no $H$-paths between vertices
in $K$, it follows that there are neither $H$-walks between
vertices of $K$, and thus, $K$ is independent by $H$-walks.
\end{proof}


\section{Panchromatic patterns by paths}
\label{sec:ppp}

We now turn our attention to
$\tilde{\mathscr{B}}_3$.   In
\cite{galeanaJGT90}, a characterization
of panchromatic patterns (by walks) was
given in terms of forbidden subdigraphs
and, although it is based on the
characterization of $\mathscr{B}_3$ found
in \cite{galeanaDM339} (which we mentioned
earlier is flawed), a similar approach can
be used to describe the structure of
the panchromatic patterns by paths.
The idea is to classify all the patterns
on three vertices; knowing which of them
are panchromatic patterns by paths gives
us enough information to describe the general
structure of patterns in this family.

We begin our analysis with a simple
observation relating $H$-kernels with
$H$-kernels by walks.   It is a direct
consequence of Theorem
\ref{transPathIffWalk}, so its proof
will be omitted.

\begin{proposition}\label{transitiveCase}
If $H$ is a transitive reflexive digraph,
then $H$ is a panchromatic pattern by paths
if and only if $H$ is a panchromatic pattern
(by walks).
\end{proposition}

Arpin and Linek showed in \cite{arpinDM307}
that every reflexive digraph with order one
or two is a panchromatic pattern. These
patterns are trivially transitive, so by
Proposition \ref{transitiveCase} we have
that the panchromatic patterns with order
one or two are also panchromatic patterns
by paths.

\begin{figure}[!htb]
\begin{center}
\includegraphics[scale=.9]{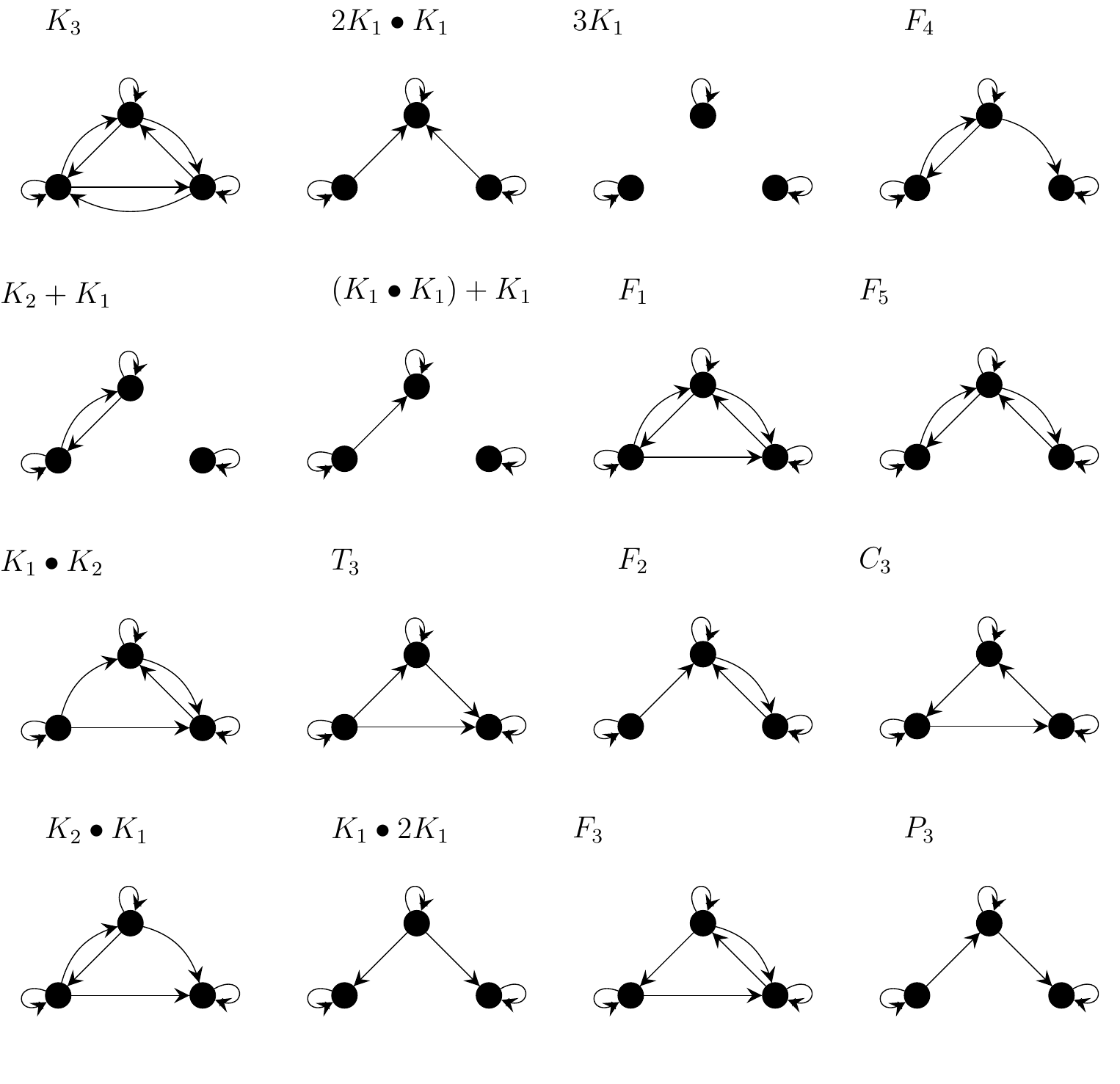}
\caption{All non-isomorphic reflexive digraphs on 3 vertices.}\label{order3}
\end{center}
\end{figure}

Using Proposition \ref{transitiveCase}
we can analyze nine of the patterns on
three vertices depicted in Figure
\ref{order3}.   The patterns $K_3$, $K_1
\bullet K_2$, $K_2 \bullet K_1$ and $K_2
+ K_1$ are transitive and belong to
$\mathscr{B}_3$, and hence, they also
belong to $\tilde{\mathscr{B}}_3$.
Analogously, the patterns $2K_1 \bullet
K_1$, $T_3$, $(K_1 \bullet K_1) + K_1$,
$K_1 \bullet 2 K_1$ and $3 K_1$ do not
belong to $\tilde{\mathscr{B}}_3$.

The following result was proved by Arpin
and Linek in \cite{arpinDM307}; although
they stated it for panchromatic patterns,
it is easy to notice that the same proof
works for panchromatic patterns by paths.

\begin{lemma}\label{arpLinNotB3}
Let $H$ be a pattern.   If $H$ contains
a walk $W$ such that $W=(x_0, x_1, \dots,
w_k)$ and
\begin{enumerate}
\item for all $x_j$, $0 \le j \le k-1$,
      there is a colour $c_j \in V_H$ such
      that $(x_j, c_j) \notin A_H$,
\item $(x_k,x_0) \notin A_H$,
\end{enumerate}
then $H\notin \tilde{\mathscr{B}_3}$.
\end{lemma}

As in \cite{arpinDM307}, Lemma
\ref{arpLinNotB3} has the following
simple consequence dealing with four
of the remaining non-transitive
patterns on three vertices.

\begin{proposition}
None of the patterns in Figure \ref{fig 5}
is in a panchromatic pattern by paths.
\end{proposition}

\begin{proof}
The digraph $D$ In Figure \ref{fig 5},
does not contain an $H$-kernel for any
of the four patterns $H$ on three vertices
found in the same figure.   It is worth
noticing that the digraph $D$ was obtained
using the construction given by the proof
of Lemma \ref{arpLinNotB3}.
\end{proof}

\begin{figure}[!htb]
\begin{center}
\includegraphics[scale=0.7]{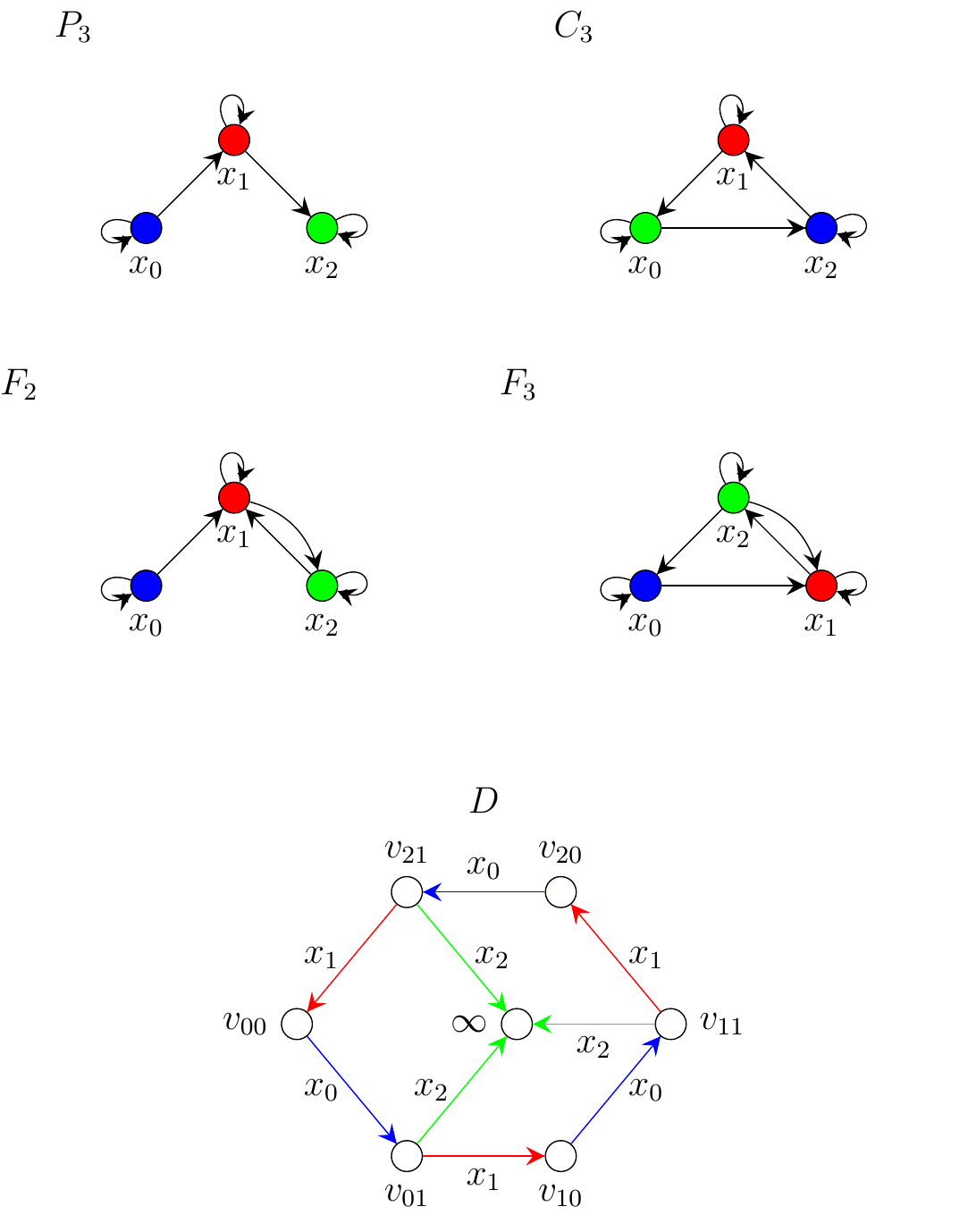}
\caption{Some patterns which are not
panchromatic patterns by
paths.}
\label{fig 5}
\end{center}
\end{figure}

Our following lemma states that
contraction of true twins and its
inverse operation, blowing up vertices
to complete reflexive digraphs,
preserve panchromatic patterns by paths.

\begin{lemma}\label{true twins}
Let $H$ be a reflexive pattern, and
let $x$ and $y$ be a pair of true
twins in $H$.   Then, $H$ is a
panchromatic pattern by paths if
and only if $H_{xy}$ is.
\end{lemma}

\begin{proof}
Since $H_{xy}$ is isomorphic to $H - y$,
an induced subdigraph of $H$, it follows
from Lemma \ref{prop Bi} that if $H \in
\tilde{\mathscr{B}_3}$, then, also $H_{xy}
\in \tilde{\mathscr{B}}_3$.

So, assume that $H_{xy}$ is a panchromatic
pattern by paths and let $D$ be an
$H$-arc-coloured digraph with colouring $c$.
We will abuse notation to consider the
digraph $D$ with two different colourings
of its arc set, $c$ and $c_{xy}$, where the
latter is an $H_{xy}$-arc-colouring of $D$
defined as follows: \[
c_{xy}((u,v))=
\left \{ \begin{array}{ll}
c ((u,v)) & \text{ if }c ((u,v))
           \ne x \text{ and } c ((u,v))
           \ne y \\
v_{xy}    & \text{ if } c ((u,v)) = x
          \text{ or }c ((u,v))=y.
\end{array}
\right.
\]

By hypothesis, $D$ has an $H_{xy}$-kernel,
$K$. We will prove that $K$ is also an
$H$-kernel of $D$.

To verify the $H$-independence, let $u$
and $v$ be different vertices in $K$.
Proceeding by contradiction, suppose that
there is an $H$-path from $u$ to $v$ in
$D$, say $P=(u = w_0, w_1, \dots, w_n =
v)$. Since $K$ is an $H_{xy}$-kernel of
$D$, then $P$ is not an $H_{xy}$-path in
$D$.   So, it follows that there is an arc
$(w_j, w_{j+1})$ in $A(P)$ such that
$c((w_j, w_{j+1})) = x$ or $c((w_j,
w_{j+1})) = y$; assume without loss of
generality that $c ((w_j, w_{j+1})) = x$.
By definition of $c_{xy}$, we have that
$c_{xy}((w_j, w_{j+1})) = v_{xy}$. Since
$H_{xy}$ is reflexive, the length of $P$
must be greater than one, otherwise $P$
would be an $H_{xy}$-path in $D$.  If $j
> 0$, then there is a vertex $z$ in $V_H$
such that $(z, x) \in A_H$ and
$(c((w_{j-1}, w_j)), c((w_j, w_{j+1})))
= (z, x)$.   So, by the definition of
$H_{xy}$ we have that $(c_{xy}((w_{j-1},
w_j)), c_{xy}((w_j, w_{j+1}))) = (z,
v_{xy}) \in A(H_{xy})$.   Analogously,
if $j < n-1$, there is $z'$ in $V_H$
such that $(c_{xy}((w_j, w_{j+1})),
c_{xy}((w_{j+1}, w_{j+2}))) = (v_{xy},
z') \in A(H_{xy})$.   Notice that either
$z$ or $z'$ could be $x$ or $y$, but
$(v_{xy}, v_{xy})$ is also an arc of
$H_{xy}$.   Therefore, we conclude that
$P$ is an $H_{xy}$-path in $D$,
contradicting that $K$ is an $H_{xy}$-kernel
of $D$. Therefore, $K$ is $H$-independent.

To prove that $K$ is $H$-absorbent, let
$u$ be a vertex in $V_D-K$. Since $K$ is
an $H_{xy}$-kernel of $D$, then there
exists $v \in K$ such that there is an
$H_{xy}$-path from $u$ to $v$ in $D$, say
$P = (u = w_0, w_1,\dots,w_n = v)$. We
will show that $P$ is also an $H$-path
of $D$. If $c_ {xy}((w_j, w_{j+1}))
\ne v_{xy}$, for $0 \le j \le n-1$,
then $P$ is trivially an $H$-path of
$D$. Suppose there exists $(w_i, w_{i+1})
\in A(T) $ such that $ c_{xy}((w_i,
w_{i+1})) = v_{xy}$. By definition of
$c_ {xy}$, we have that $c ((w_i, w_{i+1}))
= x$ or $c ((w_i, w_{i+1})) = y$, but
since $x$ and $y$ are true twins, we can
assume without loss of generality that
$c((w_i, w_{i+1})) = x$. Since $P$ is an
$H_{xy}$-path of $D$ and by the definition
of $H_{xy}$, it follows that, if $i > 0$,
then $(c ((w_{i-1}, w_i)), c((w_i, w_{i+1})
= x))$ is an arc of $H$ and if $i < n-1$,
then $(x = c ((w_i, w_{i+1})), c((w_{i+1},
w_{i+2})))$ is an arc of $H$. Therefore,
$P$ is an $H$-path of $D$ and hence, $K$
is an $H$-absorbent set of $D$.

We conclude that $K$ is an $H$-kernel of $D$.
\end{proof}

So far, we have dealt with thirteen of the
sixteen reflexive patterns on three vertices.
The remaining three patterns are harder to
deal with.   The following two sections are
devoted to patterns $F_4$, $F_5$ and $F_1$.


\section{The patterns $F_4$ and $F_5$}
\label{sec:F4F5}

In \cite{galeanaDM339}, it is proved that
patterns $F_4$ and $F_5$ are not panchromatic
patterns.   Nonetheless, while adapting the
proof for the $H$-kernel case, we noticed that
the proof of the key lemma to prove this fact
in \cite{galeanaDM339} is flawed.   At this
moment it is neither known whether $F_4 \in
\mathscr{B}_3$ nor whether $F_5 \in
\mathscr{B}_3$.   Despite this fact, a techique
similar to the one proposed in
\cite{galeanaDM339} can be used to prove that
$F_4, F_5 \notin \tilde{\mathscr{B}_3}$.

\begin{figure}[!htb]
\begin{center}
\includegraphics[scale=0.8]{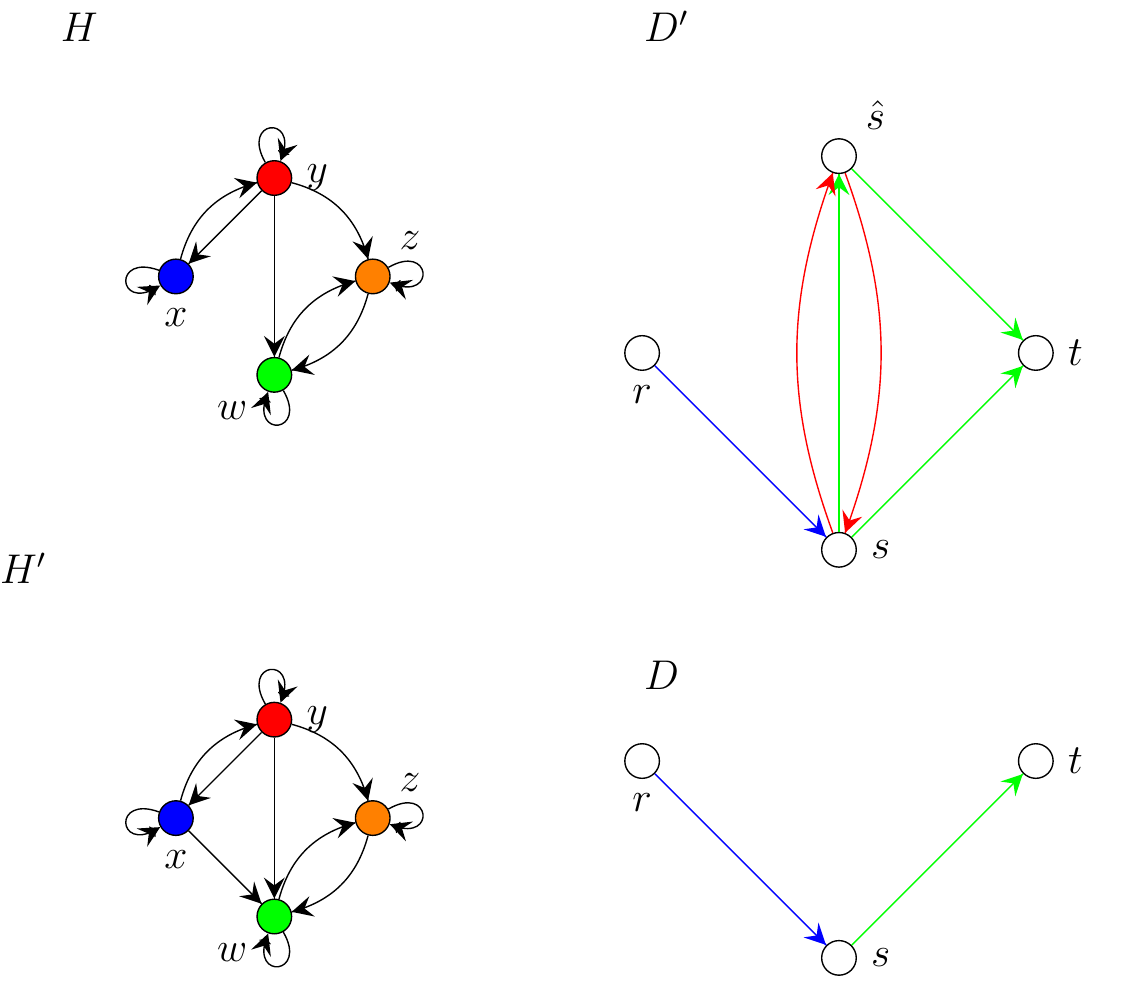}
\caption{The patterns $H$ and $H'$ in Theorem
         \ref{F4}.}
\label{FigHandH'}
\end{center}
\end{figure}

\begin{lemma} \label{F4}
The pattern $H$ in Figure \ref{FigHandH'}
is not a panchromatic pattern by paths.
\end{lemma}

\begin{proof}
Consider the digraphs $H$ and $H'$ as in
Figure \ref{FigHandH'}.   We start by showing
that for every $H'$-coloured digraph $D$ there
is an $H$-coloured multidigraph $D'$, such that
if $D'$ has an $H$-kernel, then $D$
has a kernel by $H'$-paths. Let $D$ be an
$H'$-coloured digraph and construct $D'$ from $D$
as follows. For each path $(r,s,t)$ in $D$ with
arcs $(r,s)$ and $(s,t)$ coloured by $x$ and $w$,
respectively, we add a new vertex $\hat{s}$, with
two arcs from $s$ to $\hat{s}$, one of them
coloured $y$ and the other coloured $w$, also,
an arc $(\hat{s},s)$ coloured $y$ and an arc
$(\hat{s},t)$ coloured $w$. Let $\hat{S}$ be
the set of all those new vertices $\hat{s}$
in $D'$. (See Figure \ref{FigHandH'}.)

Suppose $K'\subseteq V_{D'}$ is a kernel by
$H$-paths in $D'$. Consider the set $K$
defined as \[
K = ( K' \cup \setdef{s \in V_D}{\hat{s}
      \in K'} ) - \hat{S}.
\]

We affirm that $K$ is an $H'$-kernel in $D$.
For the $H'$-independence in $D$, consider
vertices $u$ and $v$ in $K$. Proceeding by
contradiction suppose that there is an
$H'$-path $P = (u = u_0, u_1, \ldots, u_n
= v)$ from $u$ to $v$ in $D$. Note that $P$
is a path in $D'$. If $P$ is not an $H$-path
in $D'$, then by construction there is at
least one subpath $(r,s,t)$ of $P$ such that
$(r,s)$ has colour $x$ and $(s,t)$ has colour
$w$. Consider the $H$-path $(r, s, \hat{s},
t)$ with $(r,s)$ coloured $x$, $(s,\hat{s})$
coloured by $y$ and $(\hat{s},t)$ coloured by
$w$. Let $P'$ be the path that results from
making such a replacement for each path with
form $(r,s,t)$ with the respective $(r, s,
\hat{s}, t)$. By construction $P'$ is an
$H$-path from $u$ to $v$ in $D'$. We have
4 cases.

{\em Case 1.} $\{ u,v\} \subseteq K' -
\setdef{s \in V_D}{\hat{s} \in K'}$.

It follows that $P$ or $P'$ is an
$H$-path between two vertices of $K'$
in $D'$, contradicting the independence
by $H$-paths of $K'$.

{\em Case 2.} $u \in \setdef{s \in V_D}{
\hat{s} \in K'}$ and $v \in K' - \setdef{
s \in V_D}{\hat{s} \in K'}$.

As $N^+_H(y) = V_H$, then adding the arc
$(\hat{s},s=u)$ with colour $y$ at the
beginning of $P$ (or $P'$) results in an
$H$-path between two vertices of $K'$ in
$D'$, contradicting the independence by
$H$-paths of $K'$.

{\em Case 3.} $u \in K' - \setdef{s \in
V_D}{\hat{s}\in K'}$ and $v \in \setdef{
s \in V_D}{\hat{s} \in K'}$.

As $N^-_H(y) = \set{x, y}$ and $N^-_H(w)
= \set{z, w}$, then adding the arc $(s,
\hat{s})$ with colour $y$, if $(u_{n-1},
u_n = v)$ is coloured $x$ or $y$, at the
end of $P$ (or $P'$), or adding the arc
$(s, \hat{s})$ with colour $w$, if
$(u_{n-1}, u_n = v)$ is coloured $z$ or
$w$, at the end of $P$ (or $P'$) results
in an $H$-path between two vertices of
$K'$ in $D'$, contradicting the
$H$-independence of $K'$.

{\em Case 4.} $\set{u, v} \subseteq \setdef{
s \in V_D}{\hat{s} \in K'}$.

As $N^+_H(y) = V_H$, $N^-_H(y) = \set{x, y}$
and $N^-_H(w) = \set{z,w}$, then adding the
arc $(\hat{s}, s = u)$ coloured $y$ at the
beginning of $P$ (or $P'$) and adding the
arc $(s, \hat{s})$ with colour $y$, if
$(u_{n-1}, u_n = v)$ is coloured $x$ or $y$,
at the end of $P$ (or $P'$), or adding the
arc $(s, \hat{s})$ with colour $w$, if
$(u_{n-1}, u_n = v)$ is coloured $z$ or $w$,
at the end of $P$ (or $P'$) results in an
$H$-path between two vertices of $K'$ in
$D'$, contradicting the $H$-independence $K'$.

To show that $K$ is $H'$-absorbent in $D$.
Consider $u \in V_D - K$. By definition of
$K$, we have $u \in V_{D'} - K'$. Since $K'$
is an $H$-kernel in $D'$ there is
$v\in K'$ such that $u$ reaches $v$ by
$H$-paths in $D'$. Let $P = (u = u_0, u_1,
\ldots, u_n = v)$ be an $H$-path in $D'$.
Suppose that $v \notin \hat{S}$. If $\hat{s}
\in V(P)$, then $\hat{s} = u_i$ for some
$i \in \set{1, \ldots, n-1}$, by construction
of $D'$, it follows that $s = u_{i-1}$, then
$(u_{i-1}, u_{i+1}) \in A_D$ with colour $w$,
even more, $(u_i, u_{i+1})$ has colour $w$,
thus, since $N^-_{H'}(w) = V_{H'}$, $(u_{i-2},
u_{i-1}, u_{i+1}, u_{i+2})$ is an $H'$-path.
It follows that removing all $\hat{s}$ of
$P$ an $H'$-path from $u$ to $v$ in $D$ is
obtained. If $v \in \hat{S}$, then $u_{n-1}
\in K$ hence the above procedure gives us
an $H'$-path from $u$ to $K$ in $D$. Thus,
$K$ is absorbent by $H'$-paths in $D$ and
therefore $K$ is a kernel by $H'$-paths in
$D$.

Notice that the subdigraph of $H'$ induced
by $\set{x, z, w}$ does not belong to
$\hat{\mathscr{B}}_3$, and hence, $H'
\notin \hat{\mathscr{B}}_3$. Therefore,
there is an $H'$-coloured digraph $D$ such
that $D$ has no $H'$-kernel. Let $D'$ be
constructed as before. If $H \in
\hat{\mathscr{B}}_3$, then $D'$ has a kernel
by $H$-paths, thus by the previous argument
$D$ has a kernel by $H'$-paths, which cannot
happen. Therefore $H \notin
\hat{\mathscr{B}}_3$.
\end{proof}

We will say that a vertex $z\in V_H$ is {\em
universal} if and only if $N^-_H(z) =
N^+_H(z) = V_H$.

The following lemma is very similar in
statement and proof to the previous one.

\begin{lemma} \label{F5}
The pattern $H$ in Figure \ref{FigLemma6}
is not a panchromatic pattern by paths.
\end{lemma}

\begin{proof}
Let $H$ and $H'$ be the digraphs depicted
in Figure \ref{FigLemma6}.   As in the
previous lemma, we will show that for
every $H'$-coloured digraph $D$ there
is an $H$-coloured multidigraph $D'$,
such that if $D'$ has an $H$-kernel,
then $D$ has an  $H'$-kernel. For an
$H'$-coloured digraph $D$, construct
$D'$ from $D$ as follows. For each path
$(r, s, t)$ in $D$ with arcs $(r, s)$
and $(s, t)$ coloured $x$ and $w$,
respectively, we add a new vertex
$\hat{s}$, the arcs $(s, \hat{s})$,
$(\hat{s}, s)$ coloured $y$ and the
arc $(\hat{s}, t)$ coloured $w$. Let
$\hat{S}$ be the set of all those new
vertices $\hat{s}$ in $D'$. (See Figure
\ref{FigLemma6}.)

Suppose $K' \subseteq V_{D'}$ is an
$H$-kernel in $D'$, and consider the
set $K$ defined as \[
K = (K' \cup \setdef{s \in V_D}{\hat{s}
\in K'}) - \hat{S}.
\]

We will show that $K$ is an $H'$-kernel
in $D$. For the $H'$-independence in $D$,
consider two different vertices $u$ and
$v$ in $K$. Proceeding by contradiction
suppose that there is an $H'$-path
$P = (u = u_0, u_1, \ldots, u_n = v)$
from $u$ to $v$ in $D$. Note that $P$
is a path in $D'$. If $P$ is not an
$H$-path in $D'$, by construction
there is at least on subpath $(r, s, t)$
of $P$ such that $(r, s)$ with colour
$x$ and $(s, t)$ with colour $w$.
Consider the $H$-path $(r, s, \hat{s},
t)$ with $(r, s)$ coloured $x$, $(s,
\hat{s})$ coloured $y$ and $(\hat{s},
t)$ coloured $w$. Let $P'$ be the path
that results from replacing each of
these $(r, s, t)$ paths by the respective
$(r, s, \hat{s}, t)$. By construction
$P'$ is an $H$-path from $u$ to $v$ in
$D'$. We have 4 cases.

{\em Case 1.} $\set{u, v} \subseteq K'
- \setdef{s \in V_D}{\hat{s} \in K'}$.

It follows that $P$ or $P'$ is an
$H$-path between two vertices of $K'$
in $D'$, contradicting the independence
by $H$-paths of $K'$.

{\em Case 2.} $u \in \setdef{s \in V_D}{
\hat{s} \in K'}$ and $v \in K' - \setdef{
s \in V_D}{\hat{s} \in K'}$.

Since $y$ is a universal vertex, then
adding the arc $(\hat{s}, s = u)$ with
colour $y$ at the beginning of $P$ (or
$P'$) results in an $H$-path between
two vertices of $K'$ in $D'$,
contradicting the $H$-independence of
$K'$.

{\em Case 3.} $u \in K' - \setdef{s \in
V_D}{\hat{s} \in K'}$ and $v \in \setdef{
s \in V_D}{\hat{s} \in K'}$.

As in the previous case, $y$ is a universal
vertex, and thus, adding the arc $(s,
\hat{s})$ with colour $y$ at the end of $P$
(or $P'$) results in an $H$-path between
two vertices of $K'$ in $D'$, a
contradiction.

{\em Case 4.} $\set{u, v} \subseteq
\setdef{s \in V_D}{\hat{s} \in K'}$.

Again, $y$ is a universal vertex, so adding
the arc $(\hat{s}, s = u)$ coloured $y$ at
the beginning of $P$ (or $P'$) and adding
the arc $(s, \hat{s})$ coloured $y$ at the
end of $P$ (or $P'$) results in an $H$-path
between two vertices of $K'$ in $D'$, a
contradiction.

To show that $K$ is $H'$-absorbent in $D$,
consider $u \in V_D - K$. By definition of
$K$, we have $u \in V_{D'} - K'$. Since
$K'$ is an $H$-kernel in $D'$ there exists
$v \in K'$ such that $u$ reaches $v$ by an
$H$-path in $D'$. Let $P = (u = u_0, u_1,
\ldots, u_n = v)$ be an $H$-path in $D'$.
Suppose that $v \notin \hat{S}$. If there
is some vertex $s$ such that $\hat{s} \in
V(P)$, then $\hat{s} = u_i$ for some $i
\in \set{1, \ldots, n-1}$ and by the
construction of $D'$, it follows that $s
= u_{i-1}$, so $(u_{i-1}, u_{i+1})$ is an
arc of $D$ coloured $w$. Moreover, $(u_i,
u_{i+1})$ has colour $w$ and thus, since
$N^-_{H'}(w) = V_{H'}$, $(u_{i-2}, u_{i-1},
u_{i+1}, u_{i+2})$ is an $H'$-path. It
follows that by removing all $\hat{s}$
vertices of $P$ in this way, an $H'$-path
from $u$ to $v$ in $D$ is obtained. If $v
\in \hat{S}$, then $u_{n-1} \in K$ hence
the above procedure results in an $H'$-path
from $u$ to $K$ in $D$. Thus, $K$ is
absorbent by $H'$-paths in $D$ and
therefore $K$ is an $H'$-kernel in $D$.

Finally, proceeding by contradiction,
suppose that $H \in \tilde{\mathscr{B}}_3$.
Since $H \notin \tilde{\mathscr{B}}_3$,
then there exists an $H'$-coloured digraph
$D$ which does not have an $H'$-kernel.
Let $D'$ be constructed and coloured as
before. Since $H \in \tilde{\mathscr{B}}_3$
then $D'$ has an $H$-kernel. Therefore,
by the previous argument, $D$ would have
an ¡$H'$-kernel contradicting the second
sentence of this paragraph.

\end{proof}

\begin{figure}[!htb]
\begin{center}
\includegraphics[scale=0.8]{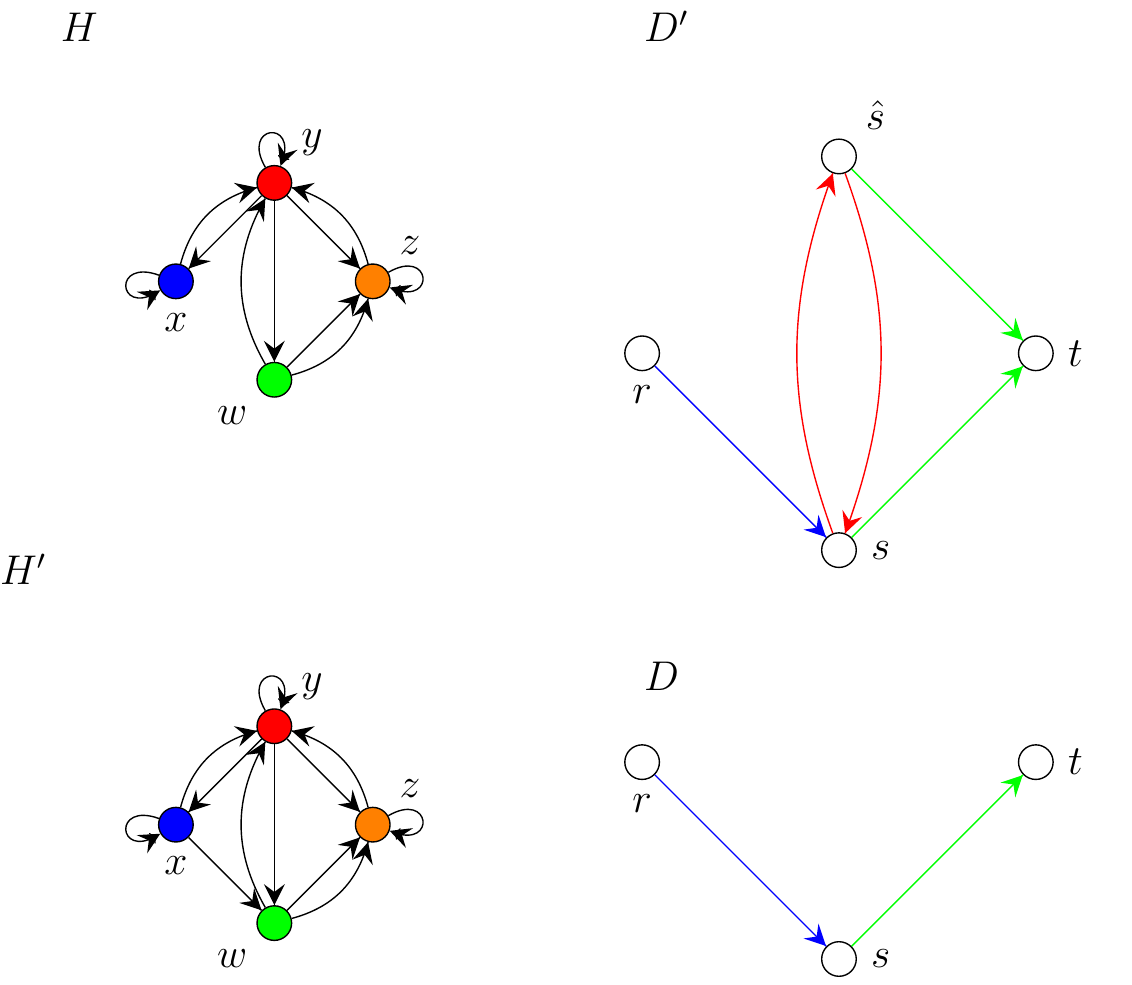}
\caption{The patterns $H$ and $H'$ in Theorem
         \ref{F5}.}
\label{FigLemma6}
\end{center}
\end{figure}

\begin{lemma} \label{notF4F5}
The patterns $F_4$ and $F_5$ (Figure
\ref{order3}) are not panchromatic
patterns by paths.
\end{lemma}

\begin{proof}
Note that $z$ and $w$ are true twins
in the pattern $H$ in Lemmas \ref{F4}
and \ref{F5}. Then, by Lemma
\ref{true twins}, the pattern $H$ is
a panchromatic pattern by paths if and
only if the pattern obtained from $H$
by contracting $z$ and $w$ is.   But
the pattern obtained by contracting
$z$ and $w$ is $F_4$ for $H$ in Lemma
\ref{F4} and $F_5$ for $H$ in Lemma
\ref{F5}.   It follows from Lemmas
\ref{F4} and \ref{F5} that $F_4$ and
$F_5$ are not panchromatic patterns
by paths.
\end{proof}


\section{The pattern $F_1$}
\label{sec:F1}

So far, for every pattern on three vertices
we know whether it is a panchromatic pattern
by paths, except for the pattern $F_1$.
Arpin and Linek proved in \cite{arpinDM307}
that $F_1$ is a panchromatic pattern (by
walks) using the concept of closure of a
coloured digraph.   Unfortunately, their
technique is neither directly applicable in
the path case, nor seems to be modifiable
to suit it.  One interesting fact that
should be pointed ot from Arpin and Linek's
technique is that they construct a
multidigraph (parallel arcs are present)
having the same reachability of a given
$F_1$-arc-coloured digraph, and where some
substructures always exist.   From here,
one may ask whether there might be an
$F_1$-arc-coloured multidigraph without an
$F_1$-kernel, while every $F_1$-arc-coloured
digraph (without parallel arcs) has an
$F_1$-kernel.   If so, then looking for a
counterexample would imply looking at
multidigraphs, which represents a broader
search space than just considering digraphs.
Fortunately, it is not the case, the main
result of this section implies that we can
only consider digraphs.

\begin{theorem} \label{F1multi}
If every $F_1$-arc-coloured digraph has an
$F_1$-kernel, then every $F_1$-arc-coloured
multidigraph has an $F_1$-kernel.
\end{theorem}

\begin{proof}
We will prove that for every
$F_1$-arc-coloured multidigraph $D$ there
is an $F_1$-arc-coloured digraph
$\widehat{D}$ such that $D$ has an
$F_1$-kernel if and only if $\widehat{D}$
has an $F_1$-kernel.   Let be $r, g$ and
$b$ the vertices of $F_1$, and assume that
the only missing arc in $F_1$ is $(b,g)$.

Let $A_D[x,y] = \setdef{f \in A_H}{f
\text{ has head } x \text{ and tail } y}$.
We construct $\widehat{D}$, an
$F_1$-arc-coloured digraph obtained from $D$
through the following modifications.
Whenever there are vertices $u$ and $v$
in $V_D$ with $|A_D[u,v]| \ge 2$, then
\begin{enumerate}
    \item If there is $f \in A_D[u,v]$
          with colour $r$, we replace
          $A_D[u,v]$ by only one arc
          from $u$ to $v$, $e_{uv}$,
          coloured $r$.

    \item If $A_D[u,v] = \set{f, f'}$ is
          such that $f$ is coloured $g$
          and $f'$ is coloured $b$, then
          we create new vertices $z_{1}$
          and $z_{2}$, add them to $V_D$,
          and create new arcs $(u, v)$,
          $(u, z_1)$, $(z_1, z_2)$ and
          $(z_1, v)$, with colours $g$,
          $b$, $g$ and $b$, respectively,
          add them to $A_D$ and delete
          $f$ and $f'$ from $A_D$.
\end{enumerate}

\begin{figure}[!htb]
\begin{center}
\includegraphics[scale=.8]{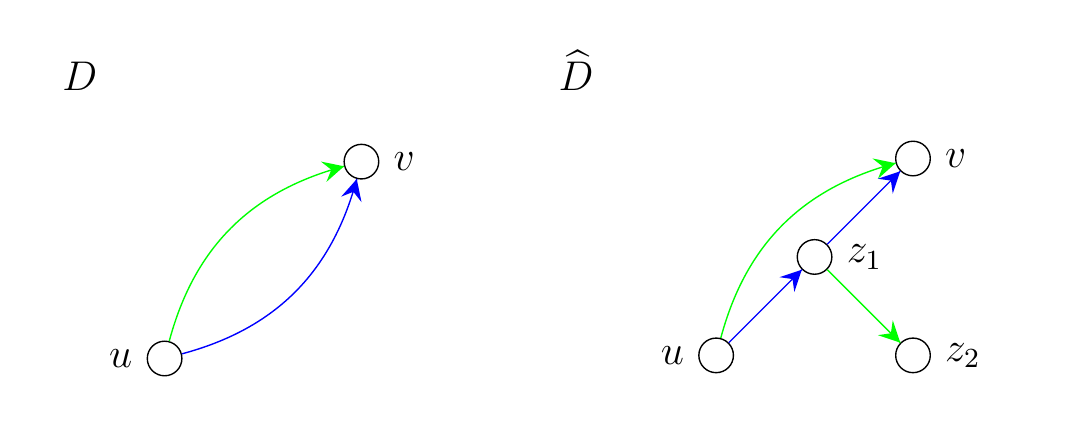}
\caption{The construction of $\widehat{D}$
in Lemma \ref{F1multi}.}
\label{DandhatD}
\end{center}
\end{figure}

Let $Z_i$ be the set of all new vertices
$z_i$ in $\widehat{D}$, $i \in \set{1, 2}$.

\textbf{Claim 1.} Let $x,y \in V_D$. There
is an $F_1$-path from $x$ to $y$ in $D$ if
and only if there is an $F_1$-path from $x$
to $y$ in $\widehat{D}$.

{\em Proof of Claim 1.} Let $W$ be the
$F_1$ path $W=(x=x_0,x_1,\ldots,x_n=y)$
in $D$. Construct $\widehat{W}$, an
$F_1$-path from $x$ to $y$ in $\widehat{D}$,
by performing the following modifications
on $W$. For each $i \in \set{0, 1, \ldots,
n-1}$ such that $|A_D[x_i, x_{i+1}]| \ge 2$,
we replace the arc from $x_i$ to $x_{i+1}$
used in $W$ with
\begin{enumerate}
    \item $e_{x_i x_{i+1}}$ coloured $r$,
          if there is $f \in A_D[x_i,
          x_{i+1}]$ coloured $r$,

    \item $(x_i, x_{i+1})$ coloured $g$,
          if there is no $f \in A_D[x_i,
          x_{i+1}]$ coloured $r$ and the
          arc from $x_i$ to $x_{i+1}$
          in $W$ is coloured $g$,

    \item $(x_i, z_1, x_{i+1})$ the
          monochromatic path coloured $b$,
          if there is no $f \in A_D[x_i,
          x_{i+1}]$ coloured $r$ and the
          arc from $x_i$ to $x_{i+1}$ in
          $W$ is coloured $b$.
\end{enumerate}

Since $W$ is an $F_1$-path in $D$, it is
enough to check that for each arc
substitution in $W$, the sequence of
colours of the new arcs are still a walk
in $F_1$. For each substitution of type 1,
$(x_{i-1}, x_i, x_{i+1}, x_{i+2})$ is an
$F_1$-path in $\widehat{D}$ because $N^+(r)
= N^-(r) = V(F_1)$ and $(x_i, x_{i+1})$
is coloured $r$ in $\widehat{W}$. For each
substitution of type 2, $(x_{i-1}, x_i,
x_{i+1}, x_{i+2})$ is an $F_1$-path in
$\widehat{D}$ because the arc from $x_i$
to $x_{i+1}$ in $W$ is coloured $g$ and
the new arc $(x_i, x_{i+1})$ in
$\widehat{W}$ is coloured $g$. For each
substitution of type 3, $(x_{i-1}, x_i,
z_1, x_{i+1}, x_{i+2})$ is an $F_1$-path
because the arc from $x_i$ to $x_{i+1}$
in $W$ is coloured $b$ and the path $(x_i,
z_1,x_{i+1})$ is monochromatic with colour
$b$ in $\widehat{D}$. Hence, $\widehat{W}$
is an $F_1$-path from $x$ to $y$ in
$\widehat{D}$.

Conversely, let $\widehat{W} = (x = x_0,
x_1, \ldots, x_n = y)$ be an $F_1$-path
from $x$ to $y$ in $\widehat{D}$ and
construct an $F_1$-path $W$ from $x$ to
$y$ in $D$ by modifying $\widehat{W}$ in
the following way. For each $i \in \set{1,
\ldots, n-1}$ such that there is $x_i
\in Z_1$, we replace the monochromatic
path $(x_{i-1}, z_1 = x_i, x_{i+1})$
coloured $b$ by the arc from $x_{i-1}$
to $x_{i+1}$ coloured $b$ in $D$, and
for each $i \in \set{0, 1, \ldots, n-1}$
such that $|A_D[x_i, x_{i+1}]| \ge 2$
different from those obtained by the
first modifications, we replace the arc
from $x_i$ to $x_{i+1}$ in $\widehat{W}$
with
\begin{enumerate}
    \item $f \in A_D[x_i, x_{i+1}]$
          coloured $r$ if the arc $x_i$
          to $x_{i+1}$ in $\widehat{W}$
          is $e_{x_i x_{i+1}}$ coloured
          $r$,

    \item the arc from $x_i$ to $x_{i+1}$
          coloured $g$ if there is no $f
          \in A_D[x_i, x_{i+1}]$ coloured
          $r$ and the arc from
          $x_i$ to $x_{i+1}$ in
          $\widehat{W}$ is coloured $g$.
 \end{enumerate}

Note that the above modifications change
a monochromatic path of colour $b$ into
an arc coloured $b$ and replace arcs by
other arcs with the same colour. Therefore,
it follows that $W$ is an $F_1$-path from
$x$ to $y$ in $D$. This ends the proof of
Claim 1.

Let $K$ be an $F_1$-kernel in $D$. We
consider $\widehat{K} = K \cup Z_2$,
we will prove that $\widehat{K}$ is an
$F_1$-kernel in $\widehat{D}$.
Notice that $\widehat{K} \cap Z_1 =
\varnothing$ and $Z_2$ is reached by
$F_1$-paths only by the vertices in
$Z_1$. Also, each vertex in $Z_2$
is a sink in $\widehat{D}$, so there
are no $F_1$-paths that start in $Z_2$.
By the $F_1$-independence of $K$ and
Claim 1, there is no $F_1$-path in
$\widehat{D}$ between vertices of $K$.
Since $Z_1 \cap K = \varnothing$ and the
observations at the beginning of this
paragraph, there are neither $F_1$-paths
in $\widehat{D}$ between vertices of $K$
and vertices of $Z_2$, nor between
vertices of $Z_2$. Hence $\widehat{K}$
is $F_1$-independent in $\widehat{D}$.

Let $x$ be a vertex in $V_{\widehat{D}}
- \widehat{K}$. If $x \in Z_1$, then
there is $z_2 \in Z_2$ such that $(x,
z_2) \in A_{\widehat{D}}$. If $x \in
V_D-K$, then there is $y\in K$ such
that there exist an $F_1$-path from $x$
to $y$ in $D$, by Claim 1 there is an
$F_1$-path from $x$ to $y$ in
$\widehat{D}$. Hence, $\widehat{K}$
is absorbent by $F_1$-paths in
$\widehat{D}$, and therefore
$\widehat{K}$ is an $F_1$-kernel in
$\widehat{D}$.

Let $\tilde{K}$ be an $F_1$-kernel in
$\widehat{D}$. Consider $K = \tilde{K}
\cap V_D$, we will prove that $K$ is
an $F_1$-kernel in $D$. Since every
vertex in $Z_2$ is a sink in
$\widehat{D}$, then $Z_2 \subset
\tilde{K}$ and consequently $Z_1
\cap \tilde{K} = \varnothing$,
moreover, the only vertices that
reach $Z_2$ by $F_1$-paths in
$\widehat{D}$ are the vertices in
$Z_1$. Since $\tilde{K}$ is
$F_1$-independent in $\widehat{D}$,
then by Claim 1, $K = \tilde{K}
\cap V_D$ is an $F_1$-independent
set in $D$.   Hence, if $x \in V_D
- K$, then by the $F_1$-absorbency
of $\tilde{K}$ in $\widehat{D}$, we
have that $K$ is an $F_1$-absorbent
set in $D$.   Therefore $K$ is an
$F_1$-kernel in $D$.

The contrapositive of the statement
of the theorem now follows directly
using the previous construction.
\end{proof}

The following corollary is immediate.

\begin{corollary} \label{F1}
If every $F_1$-arc-coloured digraph
(without parallel arcs) has an
$F_1$-kernel, then $F_1$ is a
panchromatic pattern by paths.
\end{corollary}

\section{Conclusions}
\label{sec:Conc}

To round up our comparisson between
reachability by $H$-paths and reachability
by $H$-walks, let us recall that Arpin and
Linek proved in \cite{arpinDM307} that there
is a pattern in $\mathscr{B}_2$ which is not
in $\mathscr{B}_3$.   It is easy to see that
the proof of this result is analogous for
paths (see Figure \ref{fig 7}).  So we obtain
the following result.

\begin{lemma}
The digraph $H$ such that $V_H = \set{u, u',
b, g}$ and $A(H_2) \subseteq V(H_2) - \set{
(b, u), (g, u'), (b, g), (g, b)}$ is not a
panchromatic pattern by paths.
\end{lemma}

\begin{figure}[!htb]
\includegraphics[scale=.8]{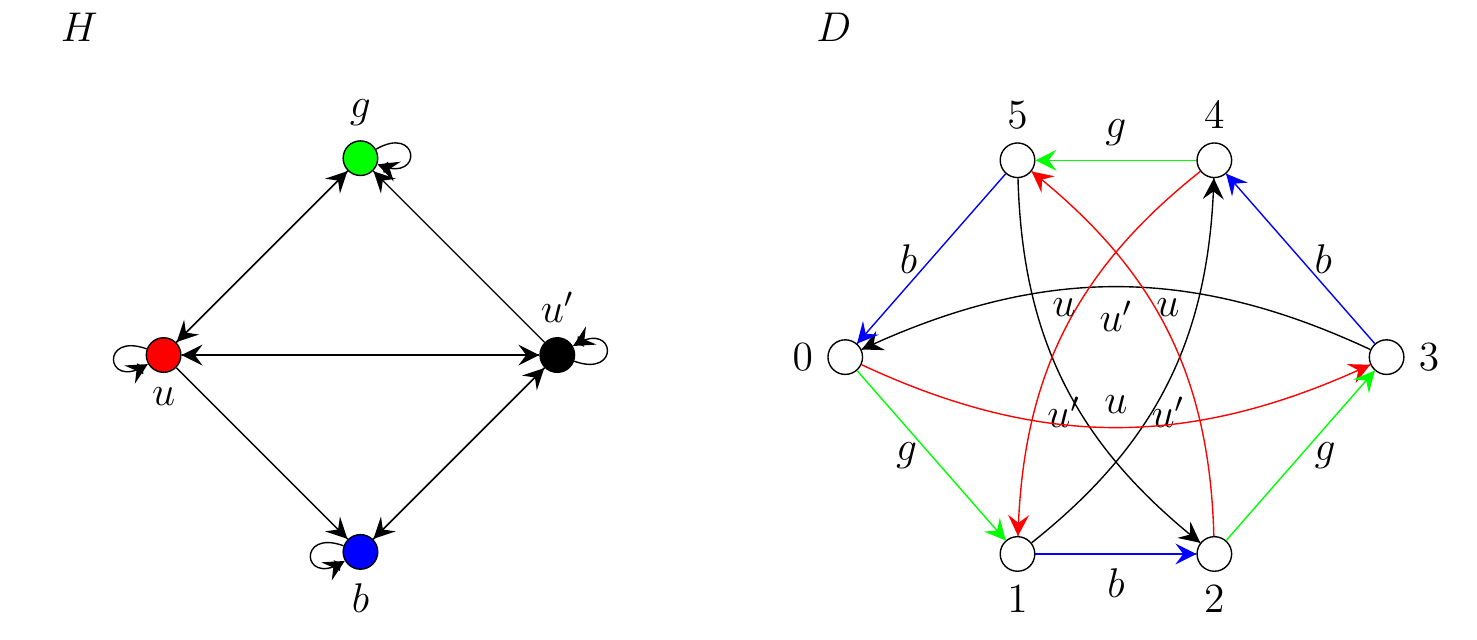}
\caption{$H \in \tilde{\mathscr{B}}_2
         - \tilde{\mathscr{B}}_3$}
\label{fig 7}
\end{figure}

Observe that $H_2^c$ has no odd cycles,
then $H_2 \in \tilde{\mathscr{B}}_2 -
\tilde{\mathscr{B}}_3$. Therefore
$\tilde{\mathscr{B}}_3 \subset
\tilde{\mathscr{B}}_2$.

So, we know exactly which of the reflexive
patterns on three vertices are panchromatic
patterns by paths, except for $F_1$.
Despite the missing pattern, we have enough
information to state the following theorem.

\begin{theorem}
Let $H$ be a pattern. Then $H$ is a
panchromatic pattern by paths if and only if
$V_H$ admits a partition $V_H = (V_1, V_2)$
such that $V_1$ and $V_2$ induce reflexive
complete digraphs (every symmetric arc is
present) and exactly one of the following
holds
\begin{itemize}
    \item There are no arcs between vertices
          in $V_1$ and vertices in $V_2$.

    \item Every vertex in $V_1$ dominates
          every vertex in $V_2$, and no
          vertex in $V_2$ dominates a vertex
          in $V_1$ (if $F_1$ is not a
          panchromatic pattern by paths).

    \item Every vertex in $V_1$ dominates
          every vertex in $V_2$ and vertices
          in $V_2$ may dominate vertices in
          $V_1$ (if $F_1$ is a panchromatic
          pattern by paths).
\end{itemize}
\end{theorem}

\begin{proof}
If $F_1$ is a panchromatic pattern by paths,
the proof of this result is found in
\cite{galeanaJGT90}.   Otherwise, the proof
is analogous (cf. \cite{hellArXiv}).
\end{proof}

We would like to point out two open problems
regarding panchromatic patterns by paths.

\begin{problem} \label{prob:F1}
Determine whether $F_1$ is a panchromatic
pattern by paths.
\end{problem}

\begin{problem} \label{prob:F4F5}
Find an $H$-arc-coloured digraph without
an $H$-kernel for $H \in \set{F_4, F_5}$.
\end{problem}

A solution to problem \ref{prob:F1} would
settle the characterization of panchromatic
patterns by paths, and Corollary \ref{F1}
guarantees that we can restrict either a
proof or the search for an example to
$F_1$-arc-coloured digraphs (without
parallel arcs).   Regarding problem,
\ref{prob:F4F5}, recall that the proofs of
Lemmas \ref{F4} and \ref{F5} are not
constructive, so, we do not have a single
example of an $H$-arc-coloured with the
desired property.   It would be really
nice to see such and example, for the sake
of completness of the subject.

\end{document}